\documentclass[a4paper, english, cleveref, autoref]{lipics-v2019}

\usepackage{complexity}
\newclass{\EnumP}{EnumP}
\newclass{\DSECP}{DSECP}
\newcommand{\CT}{\texttt{colored-tree}\xspace}

\usepackage{MnSymbol} 

\newcommand{\structure}[1]{\ensuremath{\left\langle#1\right\rangle}\xspace}

\newcommand{\set}[1]{\ensuremath{\left\{#1\right\}}\xspace}

\usepackage{listings}
\lstdefinelanguage{mio}{
  morekeywords={if,then,return,else,do,for,end ,int,False, procedure,in}
}
\makeatletter

\newcommand{\ie}{\emph{i.e.}\@\xspace}

\newcommand{\wrt}{\emph{w.r.t.}\@\xspace}

\makeatother

\newcommand{\N}{\mathbb{N}\xspace}

\DeclareMathOperator{\lcm}{lcm}
\definecolor{bleu}{rgb}{0, 0.6, 0.8}
\definecolor{rose}{rgb}{0.8, 0, 0.4}
\definecolor{vert}{rgb}{0, 0.6, 0.4}
\definecolor{yellow}{rgb}{0.6, 1, 0.75}
\definecolor{orange}{rgb}{0.7, 0.8, 0.5}

\usepackage[table]{xcolor}
\usepackage{pgf, tikz}

\usepackage{amsmath}

\usepackage{blkarray}

\usepackage{multirow}

\usepackage{tabu}
\usepackage{booktabs}
\usepackage{complexity}
\usepackage[table]{xcolor}  

\usepackage{algorithmic}
\usepackage[ruled,vlined]{algorithm2e}
\usepackage[font=small,labelfont=bf]{caption}

\usepackage{thmtools}
\usepackage{thm-restate}

\definecolor{bleu}{rgb}{0, 0.6, 0.8}
\definecolor{rose}{rgb}{0.8, 0, 0.4}
\definecolor{vert}{rgb}{0, 0.6, 0.4}

\title{Solving Equations on Discrete Dynamical Systems}

\author{Alberto Dennunzio}{DISCo - Università degli Studi di Milano-Bicocca, Italy}{dennunzio@disco.unimib.it}{}{}
\author{Enrico Formenti}{Universit\'{e} C\^{o}te d'Azur, CNRS, I3S, Nice, France}{enrico.formenti@univ-cotedazur.fr}{}{}
\author{Luciano Margara}{Università degli Studi di Bologna, Campus di Cesena, Cesena, Italy}{margara@cs.unibo.it}{}{}
\author{Valentin Montmirail}{Universit\'{e} C\^{o}te d'Azur, CNRS, I3S, Nice, France}{valentin.montmirail@univ-cotedazur.fr}{}{}
\author{Sara Riva}{Universit\'{e} C\^{o}te d'Azur, CNRS, I3S, Nice, France}{sara.riva@univ-cotedazur.fr}{}{}

\authorrunning{Dennuzio et al.}


\Copyright{Dennuzio, Formenti, Margara, Montmirail and Riva}


\category{}

\relatedversion{}

\supplement{}

\acknowledgements{Sara Riva is the main contributor.}%optional

\nolinenumbers %uncomment to disable line numbering

\graphicspath{{images/}{./results/}}

\begin{document}

\maketitle

% !TEX root = DDSE-L.tex
%%%%%%%%%%%%%%%%%%%%%%%%%%%%%%%%%%%%%%%%%%%%%%%%%%%%%%%%%%%%%%%%%%%%%%%%%%%%%%%%%%%%%%%%%%

\begin{abstract}
Boolean automata networks, genetic regulation networks, and metabolic networks are just a few examples of biological modelling by discrete dynamical systems (DDS). 
A major issue in modelling is the verification of the model against the experimental data or inducing the model under uncertainties in the data. 
Equipping finite discrete dynamical systems with an algebraic structure of commutative semiring provides a suitable context for hypothesis verification on the dynamics of DDS. 
Indeed, hypothesis on the systems can be translated into polynomial equations over DDS. 
Solutions to these equations provide the validation to the initial hypothesis. 
Unfortunately, finding solutions to general equations over DDS is undecidable. 
In this article, we want to push the envelope further by proposing a practical approach for some decidable cases in a suitable configuration that we call the \textit{Hypothesis Checking}. 
We demonstrate that for many decidable equations all boils down to a ``simpler'' equation. 
However, the problem is not to decide if the simple equation has a solution, but to enumerate all the solutions in order to verify the hypothesis on the real and undecidable systems. 
We evaluate experimentally our approach and show that it has good scalability properties.
\end{abstract}

% \thispagestyle{myheadings}
% \pagestyle{myheadings}
% \markright{\tt Proceedings of CIBB 2019}%check year

%
\section{\bf Scientific Background}
%
% !TEX root = DDSE-L.tex
%%%%%%%%%%%%%%%%%%%%%%%%%%%%%%%%%%%%%%%%%%%%%%%%%%%%%%%%%%%%%%%%%%%%%%%%%%%%%%%%%%%%%%%%%%

Boolean automata networks have been heavily used in the study of systems biology \cite{Bornholdt2008,sene2012}. The main drawback of the approach by automata network is in the very first step, namely when one induces the network from the experiments. Indeed, most of the time the knowledge of the network is partial and hypotheses are made about its real structure. Those hypotheses must be verified either by further experiments or by the study of the dynamical evolution of the network compared to the expected behaviour provided by the experimental evidences.

In~\cite{dorigatti2018}, an abstract algebraic setting for representing the dynamical evolution of finite discrete dynamical systems has been proposed. 
In the following, we denote by $R$, the commutative semi-ring of the DDS.

The basic idea is to identify a discrete dynamical system with the graph of its dynamics (finite graphs having out-degree exactly $1$) and then define operations $+$ and $\cdot$ which compose dynamical systems to obtain larger ones.

Indeed, a discrete dynamical system (DDS) is a structure $\structure{\chi,f}$ where $\chi$ is a finite set called \textbf{the set of states} and $f: \chi \to \chi$ is a function called the \textbf{next state map}. Any DDS $\structure{\chi,f}$ can be identified with its \textbf{dynamics graph} which is a structure $G\equiv\structure{V,E}$ where $V=\chi$ and $E=\set{(a,b)\in V\times V,\,f(a)=b}$. From now on, when speaking of a DDS, we will always refer to its dynamics graph.

Given two DDS $G_1=\structure{V_1,E_1}$ and $G_2=\structure{V_2,E_2}$ their \textbf{sum} $G_1 + G_2$ is defined as $\structure{V_1\cupdot V_2,E_1\cupdot E_2}$, where $\cupdot$ denotes the disjoint union. The \textbf{product} $G_1\cdot G_2$ is the structure $\structure{V', E'}$ where $V'=V_1\times V_2$ and $E'=\{((a,x),(b,y))\in V'\times V', (a,b)\in E_1\text{ and }(x,y)\in E_2\}$. 
It is easy to see that $F\equiv\structure{\chi, +, \cdot}$ is a commutative semiring in which $\structure{\emptyset,\emptyset}$ is the neutral element \wrt $+$ and $\structure{\set{a}, \set{(a,a)}}$ is the neutral element \wrt $\cdot$ operation. 

Now, consider the semiring $R[X_1,X_2,\ldots,X_n]$ of polynomials over $R$ in the variables $X_i$, naturally induced by $R$. 
Let us go back to our initial motivation. Assume that some parts of the overall dynamics $a_1, a_2, \ldots, a_k$ are known, then the following equation represents a hypothesis on the overall structure of the expected dynamical system $C$ on the basis of the known data $a_1, \ldots, a_k$, where all the coefficients, variables and $C$ are DDS.
\begin{equation} 
	a_1\cdot X_1 + a_2\cdot X_2 + \ldots + a_k\cdot X_k=C
	\label{eq:problemSolvedOriginal}
\end{equation}
\textbf{The hypotheses are verified whenever the previous equation admits a solution, therefore providing a way to solve such equation can be used to check hypotheses against a given discrete dynamical system}. For the sake of clarity, we denote our unknown variables as $X_{i}$, whereas they, in fact, represent any monomial of the form $x_{i}^{w_i}$. The following fundamental result states that solving polynomial equations over DDS is not an easy task.

\begin{theorem}[\cite{dorigatti2018}]
	\label{th:fundamental}\mbox{}\\ 
	Given two polynomials $P(x_1, \ldots, x_n)$ and $Q(x_1, \ldots , x_n)$ over $R[x_1, \ldots, x_n]$, consider the following equation 
	\begin{equation}\label{eq:fullpoly}
	P(x_1, \ldots , x_n) = Q(x_1, \ldots , x_n).
	\end{equation}
	The problem of finding a solution to Equation~\ref{eq:fullpoly} is undecidable.
	Moreover, if Equation~\ref{eq:fullpoly} is linear or quadratic, then finding a solution is in \NP. 
	Finally, when $P(x) = const$, where the polynomial is in a single variable and
	all its coefficients are systems consisting of self-loops only, the equation is 
	solvable in polynomial time.
	\end{theorem}

According to Theorem~\ref{th:fundamental}, solving polynomial equations of the type $P(x_1, \ldots, x_n) = const$ is in \NP\ even for quadratic polynomials. In order to overcome this issue, one can follow at least two strategies: either further constrain the polynomials or solve approximated equations which can provide information on the real solution.

In this article, we follow the second option. Indeed, we focus on strongly connected components (SCC) of the dynamics graph. Recall that SCC represents a very important feature in finite DDS since they are the attracting sets. These sets contain the asymptotic information about system evolution.
\section{\bf Methods}
%
% !TEX root = DDSE-L.tex
%%%%%%%%%%%%%%%%%%%%%%%%%%%%%%%%%%%%%%%%%%%%%%%%%%%%%%%%%%%%%%%%%%%%%%%%%%%%%%%%%%%%%%%%%%

In the dynamics graph, each component of a system can be divided in two parts: the transient part and the periodic part, see \cite{Mortveit2008} for more details. 
A point $x \in X$ of a discrete dynamical system $\langle X,f \rangle$ belongs to a cycle if there exists a positive number $ p \in \mathbb{N}$ such that $f^{p}(x)=x$. The smallest $p$ is the period of the cycle, and $x$ is periodic. The periodic part is the set of nodes periodic. All the others nodes are transient, but in this work, X is a finite set hence any state $x$ is \textbf{ultimately periodic} and in each component of the graph there is only one cycle of length at least 1.

Every finite DDS can be described as a sum of single components, and every component can be described, for our purposes, with the length of its period (strongly connected components in dynamics graphs are cycles). The transient part of a component is not relevant for the result of the sum and product operations when the equation is over SCC.

A single component of period $p$ is denoted $C^{1}_{p}$, while $C^{n}_{p}$ means that there are $n$ components of period p in the system. Therefore, if a system is composed by $n$ components, each of period $p_i$ with $i \in \set{1, \ldots,n}$, then $\bigoplus\limits_{i=1}^n  C^{1}_{p_{i}}$ completely describes the system where $\bigoplus$ denotes the sum of components since each component has only one period (see Figure~\ref{exass}).

\begin{figure}[!ht]
	\centering
	\begin{tikzpicture}[scale=.5,-latex,auto, semithick , font=\footnotesize, state/.style ={circle,draw,minimum width=7mm}]
    	\node[state,fill=rose!40] (a) at (0,0) {$a$};
	    \node[state,fill=vert!40] (d) at (8,0) {$d$};
	    \node[state,fill=vert!40] (e) at (10,0) {$e$};
	    \node[state,fill=bleu!40] (b) at (3,0) {$b$};
	    \node[state,fill=bleu!40] (c) at (5,0) {$c$};
	    \node[state,fill=vert!40] (f) at (12,0) {$f$};
	    \path (a) edge [loop left] node {} (a);
    	\draw[->] (b) to[bend left] (c);
	    \draw[->] (c) to[bend left] (b);
	    \draw[->] (d) to[bend left] (e);
	    \draw[->] (e) to[bend left] (f);
	    \draw[->] (f) to[bend left=40] (d);
	\end{tikzpicture} 		
	\caption{a DDS with three components: $(\textcolor{rose}{C^1_1} \oplus \textcolor{bleu}{C^1_2} \oplus \textcolor{vert}{C^1_3})$ in our notation.}
	\label{exass}
\end{figure}
\begin{remark}
When a system has several components with the same period, then their representation can be added. 
As an example, we have $C^1_2 \oplus C^1_2 = C^2_2$. Otherwise, the sum $\oplus$ consists of a concatenation of components.
\end{remark}

\section{\bf Contributions}
%
% !TEX root = DDSE-L.tex
%%%%%%%%%%%%%%%%%%%%%%%%%%%%%%%%%%%%%%%%%%%%%%%%%%%%%%%%%%%%%%%%%%%%%%%%%%%%%%%%%%%%%%%%%%

From now on, $\bar R$ will indicate the restriction of $R$ to systems made by strongly connected components only.
%Due to lack of space, the formal proofs of each theorem and lemmas are accessible at: \url{https://arxiv.org/abs/1904.13115}.
First, we need to adapt the definition of product between two DDS in terms of components and their period.

\begin{proposition}\label{lem:prodm}
	Consider a system composed by $m$ components of period $p$, namely $C^m_p$, multiplied by a system with $n$ components of period $q$, namely $C^n_q$, the result of the product operation depends only on the length of the periods of the components involved according to the following formula, with $m,n\in\N$ and $m,n \geq 1$
	\begin{equation} \label{produ}
	    C^m_p \odot C^n_q =
		C^{m \times n \times \gcd(p,q)}_{\lcm(p,q)}\enspace. 
	\end{equation}
\end{proposition}
\begin{proof}
Given two discrete dynamical systems $\langle X,f \rangle$ and $\langle Y,g \rangle$, where the first system has only one component of period $p$ and the second has only one component of period $q$, let us prove that:
$$C^1_p \odot C^1_q = C^{\gcd(p,q)}_{\lcm(p,q)}$$
We know that a product operation corresponds to a Cartesian product between $X$ and $Y$ (Given two discrete dynamical systems $\langle X,f \rangle$ and $\langle Y,g\rangle$, their product is the dynamical system $\langle X \times Y,f \times g \rangle$ where $\forall (x,y) \in X \times Y, (f \times g)(x,y)=(f(x),g(y)))$. 
There are two possible cases:
	\begin{itemize}
		\setlength\itemsep{1.2em}
		\item $gcd(p,q)=1$.
			\begin{table}[h]
			\centering
				\begin{tabular}{|c|c|c|c|c|c|c|c|c|c|c|c|c}
					$x_1$ & $x_2$ & ... & $x_p$     & $x_1$ & $x_2$ & ... & $x_p$     & ... & $x_1$ & $x_2$ & ... & $x_p$ \\
					$y_1$ & $y_2$ & ... & $y_{q_1}$ & $y_q$ & $y_1$ & ... & $y_{q-2}$ & ... & ...   & ...   & ... & $y_q$
				\end{tabular}
			\end{table}
			In this case the larger period of the two is not able to represent the smallest cyclic behavior inside it, consequently it obtains a single period containing all the Cartesian product.
		\item $gcd(p,q) \neq 1$, with $p>q$.
			\begin{table}[h]
				\centering
				\begin{tabular}{|l|l|l|l|l|l|l|}
				$x_1$ & $x_2$ & ... & $x_i$ & $x_{i+1}$ & ... & $x_p$ \\
				$y_1$ & $y_2$ & ... & $y_q$ & $y_1$     & ... & $y_q$
				\end{tabular}
			\end{table}
			In this case the cycles of period $p$ and $q$ arrive at one point to be represented by a cycle of length $lcm(p,q)$ but this means that the elements of this cycle are only a subset of the Cartesian product. For this reason $gcd(p,q)={{p \cdot q}\over{lcm(p,q)}}$ components are generated.
	\end{itemize}	
In the case of $m$ or $n$ different from $1$, this means that each product operation is done for each of these components, so in general the result is duplicated $m \cdot n$ times.
\end{proof}
One can also simplify the parameter of a component. 
The following definition provides a formula to compact the notation of a DDS with $n$ identical components.
\begin{definition}\label{lem:nxcomponent}
    Consider a single component $C_p^m$,  then $\forall n,m, p \in\N\setminus\set{0}$ it holds
    \begin{equation}\label{eq-prod}
		C^{mn}_p = n \cdot C^{m}_p \enspace.
	\end{equation}
\end{definition}

Let us remind that each $X_{i}$ represents, in fact, a variable $x_{i}^{w_i}$. 
Therefore, it is necessary to know how we can retrieve the solutions for the original $x_{i}$. 
To do so, we will use the following lemma:

\begin{proposition}
\label{general_newt}
	Given a system composed by $m$ components of period $p_i\in\N$, with $p_i>0$ for all $i \in \{1,...,m\}$, let $g(p_1,p_2,...,p_m,k_1,k_2,...,k_m)$ be the gcd between the $p_i$ for which $k_i \not = 0$ and let $l(p_1,p_2,...,p_m,k_1,k_2,...,k_m)$ be the lcm between the $p_i$ for which $k_i \not = 0$. Consider a system $S\equiv C_{p_1}^1\oplus C_{p_2}^1\oplus...\oplus C^1_{p_m}$. Then,
	\[
	\left(S\right)^n=\bigoplus\limits_{i=1}^{m}C^{p_i^{n-1}}_{p_i} \oplus \bigoplus\limits_{\substack{k_1+k_2+...+k_m=n \\ 0 \leq k_1,k_2,...,k_m < n}}^{}\binom{n}{k_1,k_2,...,k_m}C^{g\cdot{\prod_{\substack{t=1 \\k_t \neq 0}}^{m}p_t^{k_t-1}}}_{l}
	\enspace.
	\]
\end{proposition}
\begin{proof}\null\mbox{}\\
	Using the multinomial theorem one finds
	\[
	    (S)^n=(C_{p_1}^1\oplus C_{p_2}^1\oplus ...\oplus C^1_{p_m})^n=\bigoplus\limits_{k_1+k_2+...+k_m=n}^{}\binom{n}{k_1,k_2,...,k_m}\bigodot\limits_{t=1}^{m}(C^1_{p_t})^{k_t}=
	\]
	\begin{equation}\label{eq-diml2}
	=\bigoplus\limits_{i=1}^{m}(C^1_{p_i})^n\oplus \bigoplus\limits_{\substack{k_1+k_2+...+k_m=n \\ 0 \leq k_1,k_2,...,k_m < n}}^{}\binom{n}{k_1,k_2,...,k_m}\bigodot\limits_{t=1}^{m}(C^1_{p_t})^{k_t}
	\end{equation}
	The resulting Formula \ref{eq-diml2} is obtained by extrapolating the cases in which a $k_i = n$. Another transformation is possible according to Proposition \ref{lem:prodm}. 
	\[
	\bigoplus\limits_{i=1}^{m}(C^1_{p_i})^n\oplus \bigoplus\limits_{\substack{k_1+k_2+...+k_m=n \\ 0 \leq k_1,k_2,...,k_m < n}}^{}\binom{n}{k_1,k_2,...,k_m}\bigodot\limits_{t=1}^{m}(C^1_{p_t})^{k_t}=
	\]
	\[
	=\bigoplus\limits_{i=1}^{m}(C^1_{p_i})^n\oplus \bigoplus\limits_{\substack{k_1+k_2+...+k_m=n \\ 0 \leq k_1,k_2,...,k_m < n}}^{}\binom{n}{k_1,k_2,...,k_m}C^{g(p_1,p_2,...,p_m,k_1,k_2,...,k_m) \cdot \prod_{t=1}^{m} p_t^{k_t-1}}_{l(p_1,p_2,...,p_m,k_1,k_2,...,k_m)}=
	\]
	\[
	=\bigoplus\limits_{i=1}^{m}C^{p_i^{n-1}}_{p_i}\oplus \bigoplus\limits_{\substack{k_1+k_2+...+k_m=n \\ 0 \leq k_1,k_2,...,k_m < n}}^{}\binom{n}{k_1,k_2,...,k_m}C^{g(p_1,p_2,...,p_m,k_1,k_2,...,k_m) \cdot \prod_{t=1}^{m} p_t^{k_t-1}}_{l(p_1,p_2,...,p_m,k_1,k_2,...,k_m)}.
	\]
\end{proof}

For $k$ equal to $0$ we assume that $(S)^0$ is equal to $C^1_1$, the neutral element of the product operation.
Let us go back to Equation \ref{eq:problemSolvedOriginal} which is the problem that we want to solve. It can be rewritten as follows:
\begin{equation} 
	(\bigoplus\limits_{j=1}^{S_1} C_{p_{1j}}^{1} \odot X_1) \oplus(\bigoplus\limits_{j=1}^{S_2} C_{p_{2j}}^{1} \odot X_2) \oplus\ldots \oplus (\bigoplus\limits_{j=1}^{S_k} C_{p_{kj}}^{1} \odot X_k) = \bigoplus\limits_{j=1}^{m} C_{q_{j}}^{n_{j}}
	\label{eq:problemSolved}
\end{equation}
with $S_i$, the number of different components in the system $i$, $p_{ij}$ is the value of the period of the $j^{th}$ component in the system $i$. In the right term, there are $m$ different periods, where for the $j^{th}$ different period, $n_j$ is the number of components, and $q_j$ the value of the period.
However, Equation \eqref{eq:problemSolved} is still hard to solve. We can simplify it performing a \textbf{contraction step}
which consists in cutting Equation \eqref{eq:problemSolved} into two simpler equations: $(C_{p_{11}}^{1} \odot X_1) = W$, where $W \subseteq \bigoplus\limits_{i=1}^{m} C_{q_{i}}^{n_{i}}$ and $((C_1^1 \odot Y) = \bigoplus\limits_{i=1}^{m} C_{q_{i}}^{n_{i}} \setminus W)$ with $Y=(\bigoplus\limits_{i=2}^{S_1} C_{p_{1i}}^{1} \odot X_1) \oplus(\bigoplus\limits_{j=1}^{S_2} C_{p_{2j}}^{1} \odot X_2) \oplus\ldots \oplus (\bigoplus\limits_{j=1}^{S_k} C_{p_{kj}}^{1} \odot X_k)$.
By applying recursively a contraction step on all the partitions of $W$ and on the second equation obtained (\ie the one containing $Y$) one finds that, solving Equation~\eqref{eq:problemSolved} boils down to solving multiple times the following type of equation: 
\begin{equation} 
\label{eq:simple}
	C^1_p \odot X = C^n_q\enspace.
\end{equation}

If the variable $X$ presents a power different from one, it is possible use the Lemma~\ref{general_newt} in order to study the squared by the power. 

However, equations of the shape of Equation~\ref{eq:simple} will be numerous therefore an efficient practical algorithm able to enumerate all its solutions is needed. 
In fact, we can propose the following bounds to know how many times equations of the shape Equation~\ref{eq:simple} are solved with the following lemma:

\begin{proposition}
	Let us denote by $Z$ the number of times that we will solve equation of the shape Equation~\ref{eq:simple}, we have the following:
	$\prod_{i=1}^{m}\binom{n_i +\sum_{j=1}^{k}S_j-1}{\sum_{j=1}^{k}S_j-1} \cdot m \leq Z \leq \prod_{i=1}^{m}\binom{n_i +\sum_{j=1}^{k}S_j-1}{\sum_{j=1}^{k}S_j-1} \cdot m \cdot \sum_{j=1}^{k}S_j$.
\end{proposition}

The intuition is as follows: the contraction step is necessary to study all the possible ways to produce the right term with the components in the left part of the equation. Accordingly, it is necessary to understand the number of possible decompositions of the right term to discover the bounds for the number of the executions of the colored-tree method (a decomposition corresponds to assign a subset of the components of the right part to a product operation between a variable and a known component). For each period a Star and Bars decomposition is applied (we redirect the reader unfamiliar with the Star and Bars decomposition to \cite{SBars}).

\begin{proof}
	In general for a fixed $q_i$, the $n_i$ components are divided into $\sum_{j=1}^{k}S_j$ groups, in this case there are $\binom{n_i +\sum_{j=1}^{k}S_j-1}{\sum_{j=1}^{k}S_j-1}$ different ways for dividing the components. Therefore, we can rewrite the lemma as follows: $m \leq \frac{Z}{\prod_{i=1}^{m}\binom{n_i +\sum_{j=1}^{k}S_j-1}{\sum_{j=1}^{k}S_j-1}} \leq m \cdot \sum_{j=1}^{k}S_j$.
	And now, toward a contradiction for the lower bound. Let us assume that we can solve less than $m$ equations. This implies that we solve less equations than the number of different periods on the right term. Contradiction, we need at least all of them (not necessary all their combinations) to determine the solution of the equation.
	And now, toward a contradiction again to prove the upper-bound.
	Firstly, we know for all the components in the right term there are $\prod_{i=1}^{m}\binom{n_i +\sum_{j=1}^{k}S_j-1}{\sum_{j=1}^{k}S_j-1}$ feasible divisions.
	Now, let us assume that in the worst case, for each coefficient the product operation must produce more than one components of each possible periods in the right term. This is a contradiction from the definition of the equation, where all the components must all have a different period. The second possibility to go beyond this bounds is that it would exists more $S_{i}$ than the one present in the equation, again a contradiction by definition of the equation.
	Therefore, we know that we have: $\prod_{i=1}^{m}\binom{n_i +\sum_{j=1}^{k}S_j-1}{\sum_{j=1}^{k}S_j-1} \cdot m \leq Z \leq \prod_{i=1}^{m}\binom{n_i +\sum_{j=1}^{k}S_j-1}{\sum_{j=1}^{k}S_j-1} \cdot m \cdot \sum_{j=1}^{k}S_j$, for $Z$ being the number of times that we will solved equation of the shape Equation~\ref{eq:simple}.
\end{proof}
\section{\bf The Colored-Table Method}

First of all, let us formally define the problem and analyze its complexity.

\begin{definition}[\DSECP]
The (finite) Discrete Dynamical Systems Solving Equations on Components Problem is a problem which takes in input
$C^1_p$ and $C^n_q$ and outputs the list of all the solutions $X$ to the equation $C^1_p \odot X = C^n_q$.
\end{definition}

Solving \DSECP\ is hard but still tractable. Indeed, the following lemma classifies our problem in \EnumP.
Recall that \EnumP\ is the complexity class of enumeration problems for which a solution can be verified in polynomial time~\cite{phdStrozecki}. It can be seen as the enumeration counterpart of the \NP\ complexity class.

\begin{lemma}
\DSECP\ is in \EnumP.
\end{lemma}
\begin{proof}
One just needs to be able to check if a given value is a solution in polynomial time. This can be done in linear time using 
Lemma~\ref{lem:prodm}.
\end{proof}

\paragraph{Notation.} For any $n, p, q\in\N^{\star}$, let $T_{p,q}^n$ denote the set of solutions of Equation~\eqref{eq:simple}
and $S_{p,q}^n$ the set of solutions returned by the \CT method.
\medskip

%Now that we know the complexity of the problem, we can propose a practical approach to solve it. We named our approach 
%Let us turn now to the \CT\ method proposed to solve \DSECP.  
%Before going to the general case, let us first consider some sub-cases:

%\begin{itemize}
	    %\item if $p=q$ and $n=1$, then clearly $C_1^1$ is the only possible solution. 
%we must consider the divisors of q... and in general the algo is applied
%	    \item if $\gcd(p,q)=1$ and $q>p$, then there are no solutions (Lemma~\ref{lem:prodm}).
%	    \item otherwise the \texttt{colored-tree} lists all the solutions (an empty list is returned if there are no solutions). 
%\end{itemize}

The \CT method is pretty involved, we prefer start to illustrate it by an example. 
\begin{example}
	Consider the following equation $C_6^1 \odot X = C_6^6$.
	The algorithm consists in two distinct phases: tree building and solution aggregation.
	In the first phase, the algorithm enumerates all the divisors $\mathcal{D}$ of $6$ \ie $\set{6, 3, 2, 1}$. It then applies a making-change decomposition algorithm (MCDA)~\cite{AdamaszekA10} in which the total sum is $6$ and the allowed set of coins is $\mathcal{D'}=\mathcal{D}\setminus\set{6}$. MCDA decomposes $6$ as $3 + 3$ (which is an optimal decomposition).
	MCDA is then applied recursively (always using $\mathcal{D} \setminus\set{i}$ as the set of coins to decompose $i$). 
	We obtain $(6= 3+3)$, $(3= 2+1)$ and $(2= 1+1)$ as reported in Table~\ref{table:ExColoredTable}.
	\begin{table}[!b]
		\centering
		\caption{Final data-structure storing all the decompositions, each solution for each value and at each step, the set of all solutions for a given value.}\label{table:ExColoredTable}
		\scalebox{0.7}{
			\begin{tabular}{cccc}
				\toprule
				\textbf{Node} & \textbf{Splits} & \textbf{Node solution}   & \textbf{Subtree solutions set} 		  \\ \midrule 
				6    		  & [3,3][2,2,2]   	& $C_6^1$          			& 
				\begin{tabular}{@{}c@{}} 
					$\{ C^1_6,C^2_3,C^1_1 \oplus C^1_2 \oplus C^1_3,C^1_3 \oplus C^3_1,$\\
					$C^1_2 \oplus C^4_1,C^6_1,C^3_2,C^2_1 \oplus C^2_2 \}$ 
				\end{tabular}  				  \\ \rowcolor{gray!40}
				3    		  & [2,1]           & $C_3^1$          			& $\set{ C^1_3,C^1_1 \oplus C^1_2,C^3_1}$ \\ 
				2    		  & [1,1]          	& $C_2^1$          			& $\set{ C^2_1,C^1_2}$ 			  \\ \rowcolor{gray!40}
				1    		  & $\emptyset$     & $C_1^1$          			& $\set{ C_1^1}$ 				  \\
				\bottomrule
		\end{tabular}}
	\end{table}
	At this point, a check is performed to ensure that all possible ways of decomposing $6$ using $\mathcal{D'}$ are present in the tree. 
	In our case, we already have $[3,3]$ found by the first run of MCDA. We also found: $[3,2,1]$, $[2,2,1,1]$, $[1,1,2,1,1]$, $[1,1,1,1,1,1]$ by the recursive application of MCDA. By performing the check, we discover that the decomposition of $6$ as $[2,2,2]$ is not represented in the current tree. For this reason, $[2,2,2]$ is added to the set of decompositions of $6$ as illustrated in Figure~\ref{fig:treetable}, it is assigned a new color and a recursive application of MCDA is started on the newly added nodes. A new check ensures that all decompositions are
	present. This ends the building phase. The resulting tree is reported in Figure~\ref{fig:treetable}.
	\begin{figure}[!t]
		\centering
		\begin{tikzpicture}[scale=.45,-latex,auto, semithick, font=\footnotesize,
		state/.style ={circle,draw,minimum width=4mm},
		bleu/.style={fill=bleu!40},
		vert/.style={fill=vert!40}
		]
		\node[state] (a) at (0,0) {$6$};
		\node[state,bleu] (b) at (-7,-2) {$3$};
		\node[state,bleu] (c) at (-3,-2) {$3$};
		\node[state,vert] (d) at (0,-2) {$2$};
		\node[state,vert] (e) at (4,-2) {$2$};
		\node[state,vert] (f) at (7,-2) {$2$};
		
		\node[state,bleu] (g) at (-6,-4) {$1$};
		\node[state,bleu] (h) at (-8,-4) {$2$};
		\node[state,bleu] (i) at (-4,-4) {$1$};
		\node[state,bleu] (l) at (-2,-4) {$2$};
		\node[state,vert] (m) at (0,-4) {$1$};
		\node[state,vert] (n) at (2,-4) {$1$};
		\node[state,vert] (o) at (4,-4) {$1$};
		\node[state,vert] (p) at (6,-4) {$1$};
		\node[state,vert] (q) at (8,-4) {$1$};
		\node[state,vert] (r) at (10,-4) {$1$};
		
		\node[state,bleu] (s) at (-9,-6) {$1$};
		\node[state,bleu] (t) at (-7,-6) {$1$};
		\node[state,bleu] (u) at (-3,-6) {$1$};
		\node[state,bleu] (v) at (-1,-6) {$1$};
		
		\draw[->] (a) to (b);
		\draw[->] (a) to (c);
		\draw[->] (a) to (d);
		\draw[->] (a) to (e);
		\draw[->] (a) to (f);
		\draw[->] (b) to (g);
		\draw[->] (b) to (h);
		\draw[->] (c) to (i);
		\draw[->] (c) to (l);
		\draw[->] (d) to (m);
		\draw[->] (d) to (n);
		\draw[->] (e) to (o);
		\draw[->] (e) to (p);
		\draw[->] (f) to (q);
		\draw[->] (f) to (r);
		\draw[->] (h) to (s);
		\draw[->] (h) to (t);
		\draw[->] (l) to (u);
		\draw[->] (l) to (v);
		\end{tikzpicture} 
		\caption{The colored tree for the equation $C^1_6 \odot X=C^6_6$ after the completeness check.}
		\label{fig:treetable}
	\end{figure}
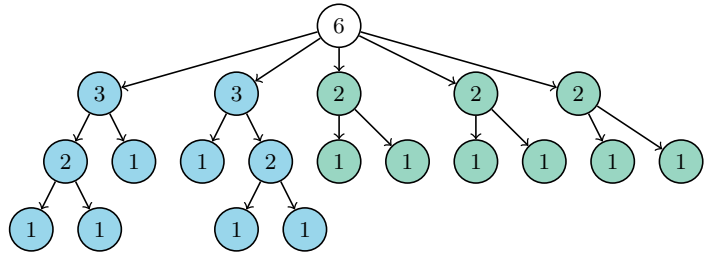
	
	After this first phase of construction of the tree, the aggregation of solutions starts. Remark that each node $m$ represents the equation $C^1_p \odot X =C^m_q$ that we call the \textbf{node equation}. The single component solution is called the \textbf{node solution} and it is obtained thanks to Lemma~\ref{lem:prodm}, $C^1_{{\frac{q}{p}} \times m}$ whenever a \textbf{feasible solution} exists \ie if $\gcd(p,{\frac{q}{p}} \times m)=m$ and $\lcm(p,{\frac{q}{p}} \times m)=q$. For example, for $m=3$ one finds $x=C_3^1$.
	To find all the solutions for the current node one must also take the Cartesian product of the solutions sets in the subtrees of the same
	color and then the union of the solution sets of nodes of different colors (different splits). All the solutions can be found in Table \ref{table:ExColoredTable}.
	%, for this reason the search of the solution starts with the computation of the single solution in each node. In other words, the problem is to find a solution for the "node equation" in which the solution is composed only by one component. As stated by Lemma \ref{lem:prodm}, a solution found is feasible iff $gcd(p,{\frac{q}{p}} \times m)=m$ and $lcm(p,{\frac{q}{p}} \times m)=q$. Hence the solution is $C^1_{{
	%\frac{q}{p}} \times m}$.
	%The last part of the method consists in the computation, for each node, of the set of solutions represented by the subtree. The aggregation of the solutions is divided in two phases. The first one is the cartesian product of the SolutionsSet between nodes of the same color (in the same split), and the second one consists in the union of the solutions provided by different colored subtrees (different splits), if more than one split exists. At the end of the execution of \texttt{ColoredTable}, the set of solutions is contained in the \emph{SolutionsSet} of the root node.
\end{example}

\begin{example}

Consider the equation $C_2^1 \odot X = C_4^5$.
In the first phase, the algorithm enumerates all the divisors $\mathcal{D}$ of $4$ \ie $\set{4, 2, 1}$. 
It then applies a making-change decomposition algorithm (MCDA)~\cite{AdamaszekA10}. MCDA decomposes $5$ as $4 + 1$ (which is an optimal decomposition).
MCDA is then applied recursively always using $\mathcal{D} \setminus\set{i}$ as the set of coins to decompose $i$. 
We obtain $(5= 4+1)$, $(4= 2+2)$ and $(2= 1+1)$ as reported in Table~\ref{table:ExColoredTablenosol}.
\begin{table}[!ht]
\centering
\scalebox{1.0}{
\begin{tabular}{cccc}
\toprule
\textbf{Node} & \textbf{Splits} & \textbf{Node solution}   & \textbf{Subtree solutions set} 		  \\ \midrule 
5    		  & [4,1]   	& $\{\}$          			& 
																	$\{\} $ 		  \\ \rowcolor{gray!40}
4    		  & [2,2]           & $\{\}$          			& $\{ C^2_4 \}$ \\ 
2    		  & [1,1]          	& $C_4^1$          			& $\{ C^1_4\}$ 			  \\ \rowcolor{gray!40}
1    		  & $\emptyset$     & $\{\}$          			& $\{\}$ 				  \\
\bottomrule
\end{tabular}}

\caption{\label{table:ExColoredTablenosol}Final data-structure storing all the decomposition, each solution for each value and at each step, the set of all solutions for a given value.}
\end{table}
At this point, a check is performed to ensure that all possible ways of decomposing $5$ using $\mathcal{D} \setminus\set{i}$ as the set of coins to decompose $i$.
In our case, we already have $[4,1]$ found by the first run of MCDA. We also found: $[2,2,1]$, $[2,1,1,1]$, $[1,1,1,1,1]$ by the recursive application of MCDA. 
By performing the check, we discover that all the possible decompositions of $5$ are represented in the current tree. 
This ends the building phase. 
The resulting tree is reported in Figure~\ref{fig:treetablenosol}.
\begin{figure}[!ht]
	\centering
	 \begin{tikzpicture}[scale=.50,-latex,auto, semithick , state/.style ={circle,draw,minimum width=0.5cm}]
	\node[state] (a) at (0,0) {$5$};
	\node[state,fill=bleu!40] (c) at (-2,-2) {$4$};
	\node[state,fill=bleu!40] (e) at (2,-2) {$1$};
	
	\node[state,fill=bleu!40] (i) at (-5,-4) {$2$};
	\node[state,fill=bleu!40] (l) at (1,-4) {$2$};
	
	\node[state,fill=bleu!40] (u) at (-1,-6) {$1$};
	\node[state,fill=bleu!40] (v) at (3,-6) {$1$};	
	\node[state,fill=bleu!40] (s) at (-3,-6) {$1$};
	\node[state,fill=bleu!40] (r) at (-7,-6) {$1$};
	
	\draw[->] (a) to[] (c);
	\draw[->] (a) to[] (e);
	\draw[->] (c) to[] (i);
	\draw[->] (c) to[] (l);
	\draw[->] (l) to[] (u);
	\draw[->] (l) to[] (v);
	\draw[->] (i) to[] (s);
	\draw[->] (i) to[] (r);
	\end{tikzpicture} 
	\caption{The tree represented in the table for $C^1_2x=C^5_4$, after the check of completeness.}
	\label{fig:treetablenosol}
\end{figure}
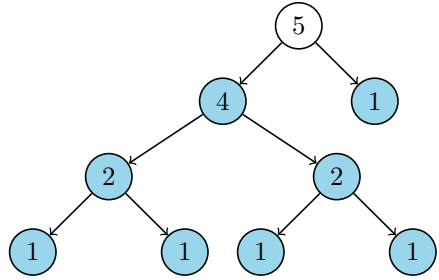
After this first phase of construction of the tree, the aggregation of solutions starts. In this case the tree presents only one color. Remark that if in the cartesian product a empty set is involved, the result of the operation is the empty set. For example, for $m=2$ ,
one has that the node solution is $C_4^1$. From the subtrees of the node one finds a empty set, but with the union of the solution of the node, the subtree solutions set for $m=2$ is $\set{C_4^1}$.
Moreover, the final solution set for the node $5$ is the empty set, in fact in the Cartesian product $m=1$ is involved (empty set). In this case the method return a empty set of solutions, that represents the impossibility of the equation.
\end{example}

\begin{example}

Consider the equation $C_2^1 \odot X = C_6^{12}$.
In the first phase, the algorithm enumerates all the divisors $\mathcal{D}$ of $6$ \ie $\set{6,3, 2, 1}$. 
It then applies a making-change decomposition algorithm (MCDA)~\cite{AdamaszekA10}. MCDA decomposes $12$ as $6 + 6$ (which is an optimal decomposition).
MCDA is then applied recursively always using $\mathcal{D} \setminus\set{i}$ as the set of coins to decompose $i$. 
We obtain $(12= 6+6)$, $(6= 3+3)$, $(3= 2+1)$ and $(2= 1+1)$ as reported in Table~\ref{table:ExColoredTablemoresol}.
\begin{table}[h!]
\centering
\scalebox{1.0}{
\begin{tabular}{cccc}
\toprule
\textbf{Node} & \textbf{Splits} & \textbf{Node solution}   & \textbf{Subtree solutions set} 		  \\ \midrule 
12    		  & [6,6]   	& $\{\}$          			& 
																	\begin{tabular}{@{}c@{}} 
																	$\{ C^4_3 \oplus C^4_6,C^{12}_3,C^6_6,C^6_3 \oplus C^3_6,$\\
																	$C^2_6 \oplus C^8_3,C^2_3 \oplus C^5_6,C^1_6 \oplus C^{10}_3\}$ 
																\end{tabular}  				  \\ \rowcolor{gray!40}
6    		  &[3,3] [2,2,2]           & $\{\}$          			& $\{ C^6_3 , C^2_6 \oplus C^2_3,C^4_3 \oplus C^1_6,C^3_6\}$ \\ 
3    		  & [2,1]           & $\{\}$          			& $\{ C^3_3,C_6^1 \oplus C_3^1 \}$ \\ \rowcolor{gray!40}
2    		  & [1,1]          	& $C_6^1$          			& $\{ C^1_6,C_3^2 \}$ 			  \\ 
1    		  & $\emptyset$     & $C_3^1$          			& $\{C_3^1\}$ 				  \\ \rowcolor{gray!40}
\bottomrule
\end{tabular}
}
\caption{\label{table:ExColoredTablemoresol}Final data-structure storing all the decomposition, each solution for each value and at each step, the set of all solutions for a given value.}
\end{table}

At this point, a check is performed to ensure that all possible ways of decomposing $12$ using $\mathcal{D'}$ is present in the tree. 
In our case, the decomposition of $6$ in $[2,2,2]$ is added in "each occurrence" of $6$. This ends the building phase. The resulting tree is reported in Figure~\ref{fig:treetablemoresol}.

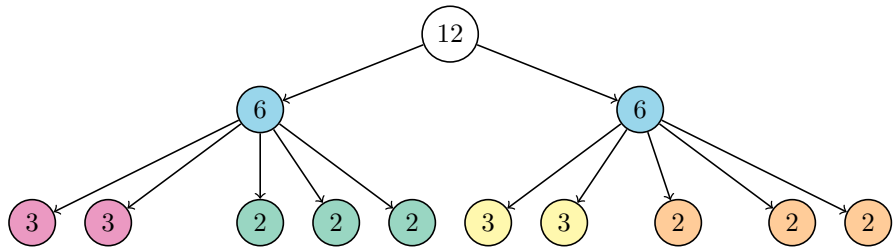
\begin{figure}[!ht]
	\centering
	 \begin{tikzpicture}[scale=.50,-latex,auto, semithick , state/.style ={circle,draw,minimum width=0.5cm}]
	\node[state] (a) at (0,0) {$12$};
	\node[state,fill=bleu!40] (b) at (-5,-2) {$6$};
	\node[state,fill=bleu!40] (c) at (5,-2) {$6$};
	
	\node[state,fill=rose!40] (d) at (-11,-5) {$3$};
	\node[state,fill=rose!40] (e) at (-9,-5) {$3$};
\node[state,fill=vert!40] (f) at (-5,-5) {$2$};
	\node[state,fill=vert!40] (g) at (-3,-5) {$2$};
\node[state,fill=vert!40] (h) at (-1,-5) {$2$};
	\node[state,fill=yellow!40] (i) at (1,-5) {$3$};
\node[state,fill=yellow!40] (l) at (3,-5) {$3$};
	\node[state,fill=orange!40] (m) at (6,-5) {$2$};
\node[state,fill=orange!40] (n) at (9,-5) {$2$};
	\node[state,fill=orange!40] (o) at (11,-5) {$2$};

%\node[state,fill=rose!40] (p) at (-18,-8) {$2$};
	%\node[state,fill=rose!40] (q) at (-16,-8) {$1$};
%\node[state,fill=rose!40] (r) at (-14,-8) {$2$};
	%\node[state,fill=rose!40] (s) at (-12,-8) {$1$};
%\node[state,fill=vert!40] (t) at (-10,-10) {$1$};
	%\node[state,fill=vert!40] (u) at (-8,-10) {$1$};
%\node[state,fill=vert!40] (v)at (-6,-10) {$1$};
%	%\node[state,fill=vert!40] (z) at (-4,-10) {$1$};
%\node[state,fill=vert!40] (aa) at (-2,-10) {$1$};
	%\node[state,fill=vert!40] (bb) at (0,-10) {$1$};
%\node[state,fill=yellow!40] (cc) at (2,-8) {$2$};
	%\node[state,fill=yellow!40] (dd) at (4,-8) {$1$};
%\node[state,fill=yellow!40] (ee) at (6,-8) {$2$};
	%\node[state,fill=yellow!40] (ff) at (8,-8) {$1$};
%\node[state,fill=orange!40] (gg) at (10,-10) {$1$};
	%\node[state,fill=orange!40] (hh) at (12,-10) {$1$};
%\node[state,fill=orange!40] (ii) at (14,-10) {$1$};
	%\node[state,fill=orange!40] (ll) at (16,-10) {$1$};
%\node[state,fill=orange!40] (mm) at (18,-10) {$1$};
	%\node[state,fill=orange!40] (nn) at (20,-10) {$1$};
	
	\draw[->] (a) to[] (c);
\draw[->] (a) to[] (b);
\draw[->] (b) to[] (d);
\draw[->] (b) to[] (e);
\draw[->] (b) to[] (f);
\draw[->] (b) to[] (g);
\draw[->] (b) to[] (h);
\draw[->] (c) to[] (i);
\draw[->] (c) to[] (l);
\draw[->] (c) to[] (m);
\draw[->] (c) to[] (n);
\draw[->] (c) to[] (o);
	\end{tikzpicture} 
	\caption{The first two levels of the tree represented in the table for $C^1_2x=C^{12}_6$, after the check of completeness.}
	\label{fig:treetablemoresol}
\end{figure}
After this first phase of construction of the tree, the aggregation of solutions starts. To find the solutions for the current node one must also take the Cartesian product of the solutions sets in the subtrees of the same
color and then the union of the solution sets of nodes of different colors (different splits). For example, for $m=12$ (\ie the root node), the cartesian product between $6$ and $6$ is computed, but for $m=6$ (in each occurrence) two cartesian operations and a union are necessary.
Therefore, the final solution set for the node $12$ is $\set{ C^4_3 \oplus C^4_6,C^{12}_3,C^6_6,C^6_3 \oplus C^3_6,C^2_6 \oplus C^8_3,C^2_3 \oplus C^5_6,C^1_6 \oplus C^{10}_3}$.
\end{example}

Although we can describe our algorithm with a pseudocode, and then we can sketch some proofs about its soundness, completeness and termination.
\begin{lstlisting}[caption=\textbf{Colored-Tree} - Complete algorithm for the enumeration problem.,
  label=coloredtable,
  escapechar=^]
procedure Colored-Tree(p, n, q):
	// input 'p,q,n': the parameters of the equation
	// enumerate all the solutions of the equation
	node,splits,nodeSolution,SubTreeSolutions=[]
	D=divisors(q) 
	node.add(n,1)
	for i in node.length do
		if (node[i]!=1) then
			splits[i]=MCDA(node[i],D \ node[i])
			generateNewNodes(splits[i])    
			SubTreeSolutions[i].add(nodeSolutions[i])
		end
	end
	checkRepresented()
	for i in node.length do
		nodeSolution[i]=computeSingleSolution(node[i])  
	end
	IncreaseOrder()
	for i in node.length do
		if (node[i]!=1) then
			solutionsSplits=[]
			for j in splits[i] do
				solutionsSplits.add(cartesian(splits[i][j]))
			end
			SubTreeSolutions[i].add(union(solutionsSplits))
		end
	end 
	return SubTreeSolutions[node.length]
\end{lstlisting}
The Lisiting \ref{coloredtable} presents the procedure using some particular functions:
\begin{itemize}
\item \textbf{generateNewNodes} adds the elements of the split, the node necessary in order to decompose but not yet represented as nodes in the nodes set.
\item \textbf{MCDA} computes the optimal solutions of the making-change problem for a node value and a set of coins.
\item \textbf{computeSingleSolution} returns the node solution for a node equation represented with a node.
\item \textbf{checkRepresented} check if all the possible decomositions of the root are represented, otherwise add the corrisponding sub-tree.
\item \textbf{IncreaseOrder} permutes the row of the table in the increasing order according to the value of the nodes.
\end{itemize}
Now we can sketch some proofs about its soundness, completeness and termination.
\begin{proposition}[Soundness]
	For all $n,p,q\in\N^{\star},\;S_{p,q}^n\subseteq T_{p,q}^n$.
\end{proposition}
\begin{proof}
Let us prove the soundness by induction on the depth of the tree from leaves to root.
\textit{Induction base}: if there is only one step, we know by Lemma~\ref{lem:prodm}, that a solution found is feasible iff $gcd(p,{\frac{q}{p}} \times m)=m$ and $lcm(p,{\frac{q}{p}} \times m)=q$, and because there is only one leaf in the base, we therefore, obtain all the solutions.
\textit{Induction hypothesis}: let us assume that we have all the possible solutions at a depth $n$ and let us show that we can obtain all the solutions at a depth $n+1$.
\textit{Induction step}: It is easy to see that a solution exists if and only if it comes from a decomposition.
Thus, by performing a Cartesian product between the set of solutions at depth $n$ (which is true by IH) and the node solution (which is true by Induction base, since the node can be seen as a leaf), we know that we will obtain all the solution coming from the possible decomposition in the sub-tree.
If a solution is coming from another sub-tree, since we perform an exhaustive check where we assign a different color to the other sub-tree, we know again, by IH and because we are taking the union of all the possible solutions, that we have all the possible solutions at a depth $n+1$.
\end{proof}

\begin{proposition}[Completeness]
For all $n,p,q\in\N^{\star},\;T_{p,q}^n\subseteq S_{p,q}^n$.
\end{proposition}
\begin{proof}
By contradiction, let us assume that there exists a solution $r \in T_{p,q}^n$ and that $r \not\in S_{p,q}^n$. This means that the colored-tree method does not return it. This implies that it exists a decomposition of $n$, which leads to $r$, such that this decomposition is not in the tree. This is impossible since, an exhaustive check is performed to assure that all the decompositions are there. Therefore, all solutions are returned.
\end{proof}

\begin{proposition}[Termination]
	The colored-tree method always terminates.
\end{proposition}
\begin{proof}
	The building phase always terminates since the colored-tree has maximal depth $\mathcal{D'}=div(q,n)$ and the number of different
	possible colors is bounded by $2^k$ where $k$ is the size of the multi-set containing $n/p_i$ copies of the divisor $p_i$ per each divisor in $\mathcal{D'}$. The aggregation phase always terminates since it performs a finite number of operations per each node of the colored tree.
\end{proof}

%Due to lack of space, we do not formally prove the termination, but it easily comes from the finite number of divisors.
Now that we have defined the problem, its complexity and a sound and complete algorithm to solve it. 
It is time to experimentally evaluate it in order to study its scalability.

\section{\bf Experimental Evaluations}

% !TEX root = DDSE-L.tex
%%%%%%%%%%%%%%%%%%%%%%%%%%%%%%%%%%%%%%%%%%%%%%%%%%%%%%%%%%%%%%%%%%%%%%%%%%%%%%%%%%%%%%%%%%

The \CT method provides a complete set of solutions of simple equations of type Equation~\ref{eq:simple}. 
Its complexity can be experimentally measured counting the number of nodes in the colored tree.

\begin{figure}[]
\centering
%\begin{minipage}{0.49\textwidth}
	\includegraphics[width=0.65\textwidth]{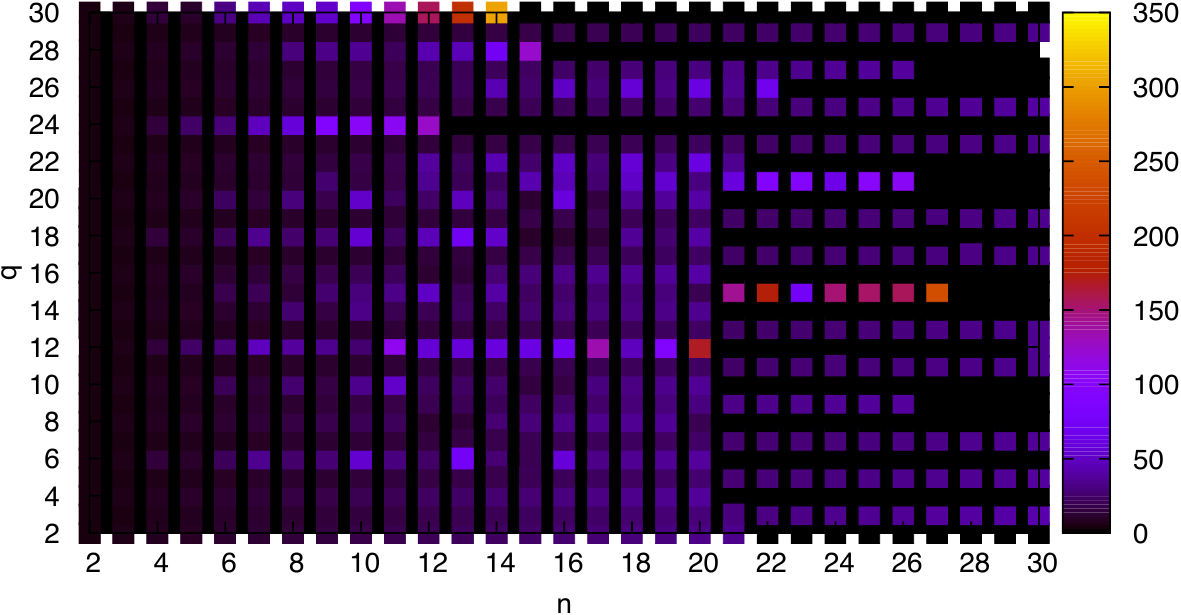}
	\caption{The number of nodes in the colored tree as a function of $n$ and $q$.}
	\label{nnode}
%\end{minipage}
\end{figure}

Figure~\ref{nnode} shows how the complexity grows as a function of $n$ and $q$. For this case, we set $p=q$ to ensure that we always have at least one solution and therefore a tree-decomposition. Notice that, in some cases, the complexity is particularly high due to specific analytical relations between the input parameters that we are going to study in the future. Notice also that our method seems to have a weakness when $q$ is an even number. This is easily explained: in many cases, all the divisors can be expressed by the other ones. Therefore the check that ensures that all the decompositions are present is particularly time- and memory-consuming.

\begin{figure}[!t]
 	\centering
	\includegraphics[width=0.65\textwidth]{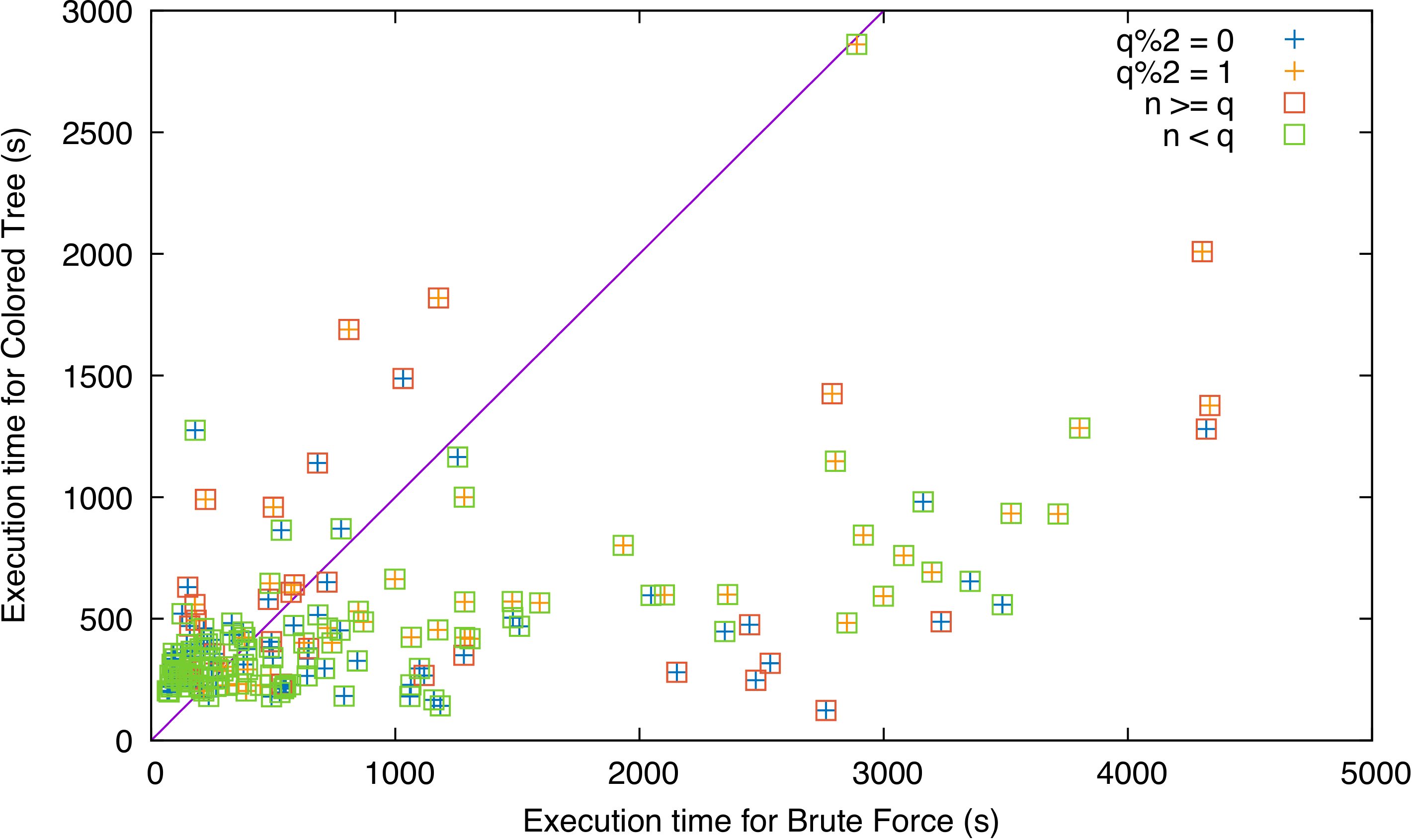}
	\caption{The brute force approach vs. \CT method \wrt execution time (in seconds).}
	\label{time}
%\end{minipage}
\end{figure}
Since there is no other competitor algorithm at the best of our knowledge, we compared the \CT method to a brute force algorithm. 
We test our algorithm with $n$ from $1$ to $20$, $p$ is also from $1$ to $20$ and at any time, $p=q$.
Results are reported in Figure~\ref{time}. As expected, the \CT method outperforms the brute force solution, sometimes with many orders of magnitude faster.
However, when the input equation has small coefficients, the \CT method performs worse. 
This can be explained considering that building the needed data structures requires a longer time than the execution of the brute force algorithm.
\section{\bf Conclusion}
%
% !TEX root = DDSE-L.tex
%%%%%%%%%%%%%%%%%%%%%%%%%%%%%%%%%%%%%%%%%%%%%%%%%%%%%%%%%%%%%%%%%%%%%%%%%%%%%%%%%%%%%%%%%%

Questions about boolean automata networks, used in biological modelling for genetic regulatory networks and metabolic networks, can be rewritten as equations over DDS using the formalism introduced in Dennunzio et al. in \cite{dorigatti2018}.
They argued that polynomial equations are a convenient tool for the analysis of the dynamics of a system. However,
algorithmically solving such equations is an unfeasible task. In this article, we propose a practical way to 
partially overcome those difficulties using a couple of approaches which aims at studying separately the number of
component (\ie the number of attractors) and the length of their periods. This paper proposes an algorithm for the
number of components of the solution of a polynomial equation over finite DDS.

One of the core routines of the algorithm uses a brute force check for the make-change problem which clearly affects the overall performances. Therefore, a natural research direction consists in finding a better performing routine. One possibility would consider parallelisation since a large part of the computations are strictly indipendent. Another interesting research direction consists inbetter understanding the computational complexity of the \DSECP.
We are still working to improve the performances of the algorithm to have stronger scalability properties in the perspective of providing a handy tool which can be exploited by bioinformaticians to actually solve the \textit{Hypothesis Checking} problem in their context.

\section*{\bf Acknowledgments}

This work has been partly funded by IDEX UCA$^{\textsc{jedi}}$.

\bibliography{DDSE-L}

\begin{thebibliography}{1}

\bibitem{AdamaszekA10}
Anna Adamaszek and Michal Adamaszek.
\newblock Combinatorics of the change-making problem.
\newblock {\em Eur. J. Comb.}, 31(1):47--63, 2010.
\newblock \href {https://doi.org/10.1016/j.ejc.2009.05.002}
  {\path{doi:10.1016/j.ejc.2009.05.002}}.

\bibitem{Bornholdt2008}
Stefan Bornholdt.
\newblock Boolean network models of cellular regulation: prospects and
  limitations.
\newblock {\em Journal of The Royal Society Interface}, 5:85--94, 2008.
\newblock URL: \url{http://doi.org/10.1098/rsif.2008.0132.focus}, \href
  {https://doi.org/http://doi.org/10.1098/rsif.2008.0132.focus}
  {\path{doi:http://doi.org/10.1098/rsif.2008.0132.focus}}.

\bibitem{dorigatti2018}
Alberto Dennunzio, Valentina Dorigatti, Enrico Formenti, Luca Manzoni, and
  Antonio~E. Porreca.
\newblock Polynomial equations over finite, discrete-time dynamical systems.
\newblock In {\em Proc. of {ACRI}'18}, pages 298--306, 2018.
\newblock \href {https://doi.org/10.1007/978-3-319-99813-8\_27}
  {\path{doi:10.1007/978-3-319-99813-8\_27}}.

\bibitem{SBars}
Oscar Levin.
\newblock {\em Discrete Mathematics: an open introduction}.
\newblock https://github.com/oscarlevin/discrete-book/, 2017.

\bibitem{Mortveit2008}
Henning~S. Mortveit and Christian~M. Reidys.
\newblock {\em An Introduction to Sequential Dynamical Systems}.
\newblock Universitext. Springer, 2008.

\bibitem{sene2012}
Sylvain Sen{\'e}.
\newblock {\em {On the bioinformatics of automata networks}}.
\newblock Habilitation {\`a} diriger des recherches, {Universit{\'e} d'Evry-Val
  d'Essonne}, November 2012.
\newblock URL: \url{https://tel.archives-ouvertes.fr/tel-00759287}.

\bibitem{phdStrozecki}
Y.~Strozecki.
\newblock {\em Enumeration complexity and matroid decomposition}.
\newblock PhD thesis, Universit\'e Paris Diderot - Paris 7, 2010.

\end{thebibliography}

\end{document}